\documentclass[a4paper,12pt]{article}
\usepackage[T2A]{fontenc}                       
\usepackage[cp1251]{inputenc}            
\usepackage[all]{xy}
\usepackage[english]{babel}
\usepackage{amssymb}
\usepackage{cite}
\usepackage{cmap}
\usepackage{latexsym}
\usepackage{amsmath}
\usepackage{enumerate}
\usepackage{amsmath, amsthm, amscd, amsfonts, amssymb, graphicx, color}
 \usepackage[left=2.5cm,right=1.5cm,top=1.3cm,bottom=2cm]{geometry}
\usepackage{indentfirst}

\newtheorem{thm}{Theorem}[section]
\newtheorem{cor}[thm]{Corollary}
\newtheorem{lem}[thm]{Lemma}
\newtheorem{prop}[thm]{Proposition}

\newtheorem{defn}{Definition}[section]
\newtheorem{rem}{Remark}[section]
\newtheorem{example}{Example}[section]
\newtheorem{quest}{Question}[section]

\begin{document}

\title{On the Generalized Fitting Height and Nonsoluble Length of the Mutually Permutable Products of Finite Groups\footnote{Supported by BFFR $\Phi23\textrm{PH}\Phi\textrm{-}237$}}

\author{Viachaslau I. Murashka${}^{1,2}$ and Alexander F. Vasil'ev${}^{3,2}$}
\date{}
\footnotetext[1]{email: mvimath@yandex.ru}
\footnotetext[2]{Francisk Skorina Gomel State University, Gomel, Belarus}
\footnotetext[3]{email: formation56@mail.ru}

 \maketitle

\begin{abstract}
The generalized Fitting height $h^*(G)$ of a finite group $G$ is the least number $h$ such that $\mathrm{F}_h^* (G) = G$, where $\mathrm{F}_{(0)}^* (G) = 1$, and $\mathrm{F}_{(i+1)}^*(G)$ is the inverse image of the generalized Fitting subgroup $\mathrm{F}^*(G/\mathrm{F}^*_{(i)} (G))$.
Let $p$ be a prime, $1=G_0\leq G_1\leq\dots\leq G_{2h+1}=G$ be the shortest normal series in which for $i$ odd the factor $G_{i+1}/G_i$ is $p$-soluble (possibly trivial),
and for $i$ even the factor $G_{i+1}/G_i$ is a (non-empty) direct product of nonabelian simple
groups. Then $h=\lambda_p(G)$ is called the non-$p$-soluble length of a group $G$.
We proved that if a finite group $G$ is a mutually permutable product of of subgroups $A$ and $B$ then
 $\max\{h^*(A), h^*(B)\}\leq h^*(G)\leq \max\{h^*(A), h^*(B)\}+1$ and $\max\{\lambda_p(A), \lambda_p(B)\}= \lambda_p(G)$. Also we introduced and studied the non-Frattini length.

Keywords: Finite group; generalized Fitting subgroup; mutually permutable product of groups;
generalized Fitting height; non-$p$-soluble length; Plotkin radical.
\end{abstract}

\section{Introduction and the Main Results}

All groups considered here are finite. E.I. Khukhro and P. Shumyatsky introduced and studied interesting invariants of a group: the generalized Fitting height   and the nonsoluble length \cite{Khukhro2015a, KHUKHRO2015, Khukhro2015b}. The first one is the extension of the well known Fitting height to the class of all groups and the second one implicitly appeared in \cite{Hall1956, Wilson1983}.

\begin{defn}[Khukhro, Shumyatsky]
  $(1)$ The generalized Fitting height $h^*(G)$ of a finite group $G$ is the least number $h$ such that $\mathrm{F}_h^* (G) = G$, where $\mathrm{F}_{(0)}^* (G) = 1$, and $\mathrm{F}_{(i+1)}^*(G)$ is the inverse image of the generalized Fitting subgroup $\mathrm{F}^*(G/\mathrm{F}^*_{(i)} (G))$.

  $(2)$  Let $p$ be a prime, $1=G_0\leq G_1\leq\dots\leq G_{2h+1}=G$ be the shortest normal series in which for $i$ odd the factor $G_{i+1}/G_i$ is $p$-soluble $($possibly trivial$)$,
and for $i$ even the factor $G_{i+1}/G_i$ is a $($non-empty$)$ direct product of nonabelian simple
groups. Then $h=\lambda_p(G)$ is called the non-$p$-soluble length of a group $G$.

$(3)$  Recall that $\lambda_2(G)=\lambda(G)$ is the nonsoluble length of a group $G$.
\end{defn}

  In \cite{Khukhro2015b} E.I. Khukhro and P. Shumyatsky  showed that in the general case the generalized Fitting height of a factorized group is not bounded in terms  of the generalized Fitting heights of factors. The same situation is also for the nonsoluble length.

  Recall \cite[Definition 4.1.1]{PFG} that a group $G$ is called a mutually permutable product of its subgroups $A$ and $B$ if $G=AB$,  $A$ permutes with every subgroup of $B$ and    $B$ permutes with every subgroup of $A$. The products of mutually permutable subgroups is the very interesting topic of the theory of groups (for example, see \cite[Chapter 4]{PFG}).

  The main result of our paper is

\begin{thm}\label{Main1}
  Let a group $G$ be the product of the mutually permutable subgroups $A$ and $B$. Then

  $(1) $ $\max\{h^*(A), h^*(B)\}\leq h^*(G)\leq \max\{h^*(A), h^*(B)\}+1$.

  $(2)$  $\max\{\lambda_p(A), \lambda_p(B)\}= \lambda_p(G)$ for any prime $p$. In particular,  $\max\{\lambda(A), \lambda(B)\}= \lambda(G)$.
\end{thm}

If a group $G$ is soluble, then $h^*(G)=h(G)$ is the Fitting height of a group $G$.

\begin{cor}[\hspace{-0.7pt}\cite{Jabara2019}]\label{cor82}
  If a soluble group $G$ is the product of the mutually permuta\-ble subgroups $A$ and $B$, then $\max\{h(A), h(B)\}\leq h(G)\leq \max\{h(A), h(B)\}+1$.
\end{cor}

\begin{example}
  Note that the symmetric group $\mathbb{S}_3$ of degree $3$ is the mutually permutable product of the cyclic groups $Z_2$ and $Z_3$  of orders $2$ and $ 3$ respectively.  Hence $h^*(\mathbb{S}_3)=\max\{h^*(Z_2), h^*(Z_3)\}+1=\max\{h(Z_2), h(Z_3)\}+1$.
\end{example}

\section{The Functorial Method}

According to B.I. Plotkin \cite{Plotkin1973} a functorial   is a function $\gamma$ which assigns to each group $G$ its characteristic subgroup $\gamma(G)$   satisfying $f(\gamma(G)) =
\gamma(f(G)) $ for any isomorphism $f: G \to G^*$.

We are interested in functorials with some properties:

$(F1)$ $f(\gamma(G))\subseteq \gamma(f(G))$ for every epimorphism $f: G\to G^*$.

$(F2)$ $\gamma(N)\subseteq \gamma(G)$ for every $N\trianglelefteq G$.

$(F3)$  $\gamma(G)\cap N\subseteq\gamma(N)$ for every $N\trianglelefteq G$.

\begin{rem}\label{rem21}
  $(0)$ Functions $\mathrm{F}^*$ and $R_p$ that assign to every group respectively its the generalized Fitting subgroup and the $p$-soluble radical are examples of functorials. It is well known that they satisfy $(F1), (F2), (F3)$.

  $(1)$ Recall that a functorial $\gamma$ is called a Plotkin radical if it satisfies $(F1)$, idempotent $($i.e. $\gamma(\gamma(G))=\gamma(G))$ and $N\subseteq \gamma(G)$ for every $\gamma(N)=N\trianglelefteq G$ \cite[p. 28]{Gardner2003}.

  $(2)$ A functorial   that satisfies $(F3)$ is often called hereditary $($nevertheless, the same word means different in the theory of classes of groups$)$.

   $(3)$ A functorial $\gamma$ is a hereditary Plotkin radical if and only if it satisfies $(F1), (F2), (F3)$. Let prove it. Assume that $\gamma$ is a hereditary Plotkin radical. We need only to prove that it satisfies $(F2)$. If $N\trianglelefteq G$, then $\gamma(N)\textrm{ char }N\trianglelefteq G$. So $\gamma(N)\trianglelefteq G$. Now $\gamma(N)=\gamma(\gamma(N))\subseteq \gamma(G)$. Thus a hereditary Plotkin radical satisfies  $(F1), (F2), (F3)$. Assume that $\gamma$ satisfies  $(F1), (F2), (F3)$. We need only to prove that it is idempotent. By $(F3)$ we have $\gamma(G)=\gamma(G)\cap G\subseteq\gamma(\gamma(G))\subseteq\gamma(G)$. Thus $\gamma(\gamma(G))=\gamma(G)$.

   $(4)$ The functorial $\Phi$ which assigns to every group $G$ its  Frattini subgroup $\Phi(G)$ satisfies $(F1)$ and $(F2)$ but not $(F3)$.

   $(5)$ If $\gamma$ satisfies $(F2)$ and $(F3)$, then $\gamma(G)\cap N=\gamma(N)$ for every group $G$ and $N\trianglelefteq G$.
\end{rem}

\begin{lem}\label{FF}
  If $\gamma$  satisfies $(F1)$ and $(F2)$, then $\gamma(G_1\times G_2)=\gamma(G_1)\times \gamma(G_2)$ for any groups $G_1$ and $G_2$.
\end{lem}

\begin{proof}
  From $G_i\trianglelefteq G_1\times G_2$ it follows that $\gamma(G_i)\subseteq \gamma(G_1\times G_2)$ by $(F2)$ for $i\in\{1, 2\}$.
Note that $\gamma(G_1\times G_2)G_i/G_i\subseteq \gamma((G_1\times G_2)/G_i)=(\gamma(G_{\bar i})\times G_i)/G_i$ by $(F1)$ for $i\in\{1, 2\}$. Now
    \begin{multline*}
      \gamma(G_1\times G_2)\subseteq (\gamma(G_1\times G_2)G_2)\cap (\gamma(G_1\times G_2)G_1)\subseteq \\(\gamma(G_1)\times G_2)\cap (G_1\times\gamma(G_2))=\gamma(G_1)\times \gamma(G_2).
    \end{multline*}
    Thus $\gamma(G_1\times G_2)=\gamma(G_1)\times \gamma(G_2)$.
\end{proof}

Recall \cite{Plotkin1973} that  for functorials
  $\gamma_1$ and $\gamma_2$  the upper product $\gamma_2\star\gamma_1$ is  defined by $$(\gamma_2\star\gamma_1)(G)/\gamma_2(G)=\gamma_1(G/\gamma_2(G)).$$

  \begin{prop}\label{length0}
  Let $\gamma_1$ and $\gamma_2$ be functorials. If $\gamma_1$ and $\gamma_2$  satisfy $(F1)$ and $(F2)$, then $\gamma_2\star\gamma_1$ satisfies $(F1)$ and $(F2)$. Moreover if $\gamma_1$ and $\gamma_2$ also satisfy $(F3)$, then  $\gamma_2\star\gamma_1$ satisfies $(F3)$.
\end{prop}

\begin{proof}
  $(1)$  \emph{$\gamma_2\star\gamma_1$ satisfies $(F1)$}.

  Let $f: G\to f(G)$ be an epimorphism.  From $f(\gamma_2(G))\subseteq\gamma_2(f(G))$ it follows that the following diagram is commutative.
\[
\xymatrix{
G \ar[r]^{f} \ar[dr]^{f_4} \ar[d]^{f_1} & f(G) \ar[d]^{f_3} \\
  G/\gamma_2(G) \ar[r]^{f_2} & f(G)/\gamma_2(f(G))
}
\]
Let $X=\gamma_1(G/\gamma_2(G))$ and $Y=\gamma_1(f(G)/\gamma_2(f(G)))$. Note that $(\gamma_2\star\gamma_1)(G)=f_1^{-1}(X)$ and $(\gamma_2\star\gamma_1)(f(G))=f_3^{-1}(Y)$ by the definition of $\gamma_2\star\gamma_1$.
Since $\gamma_1$ satisfies $(F1)$, we see that $f_2(X)\subseteq Y$.  Hence $X\subseteq f_2^{-1}(Y)$. Now $(\gamma_2\star\gamma_1)(G)\subseteq   f_1^{-1}(f_2^{-1}(Y))=f_4^{-1}(Y)$. So  $$f((\gamma_2\star\gamma_1)(G))\subseteq   f(f_4^{-1}(Y))=f_3^{-1}(Y)=(\gamma_2\star\gamma_1)(f(G)).$$ Thus $\gamma_2\star\gamma_1$ satisfies $(F1)$.

$(2)$ \emph{$\gamma_2\star\gamma_1$ satisfies $(F2)$}.

    Let $N\trianglelefteq G$. From $\gamma_2(N)\textrm{ char }N\trianglelefteq G$ it follows that $\gamma_2(N)\trianglelefteq G$.  Since $\gamma_2$ satisfies $(F2)$, we see that $\gamma_2(N)\subseteq \gamma_2(G)$. So the following diagram is commutative.

\[
\xymatrix{
G \ar[r]^{f_1} \ar[dr]_{f_3}& G/\gamma_2(N) \ar[d]^{f_2} \\
 & G/\gamma_2(G)
}
\]
Let $X=\gamma_1(G/\gamma_2(N))$, $Y=\gamma_1(N/\gamma_2(N))$ and $Z=\gamma_1(G/\gamma_2(G))$. Note that      $(\gamma_2\star\gamma_1)(G)=f_3^{-1}(Z)$ and $(\gamma_1\star\gamma_2)(N)\subseteq f_1^{-1}(Y)$. Since $\gamma_1$ satisfies $(F1)$ and $(F2)$, we see that $f_2(X)\subseteq Z$ and $Y\subseteq X$. Now
     $$(\gamma_2\star\gamma_1)(N)\subseteq f_1^{-1}(Y)\subseteq f_1^{-1}(X)\subseteq f_1^{-1}(f_2^{-1}(Z))=f_3^{-1}(Z)=(\gamma_2\star\gamma_1)(G).$$
Hence $\gamma_2\star\gamma_1$ satisfies $(F2)$.

$ (3)$ \emph{If $\gamma_1$ and $\gamma_2$ also satisfy $(F3)$, then  $\gamma_2\star\gamma_1$ satisfies $(F3)$.}

 Assume that $\gamma_1$ and $\gamma_2$  satisfy $(F2)$ and $(F3)$. Let $N\trianglelefteq G$.

Since $N\gamma_2(G)/\gamma_2(G)\cap(\gamma_2\star\gamma_1)(G)/\gamma_2(G)\trianglelefteq (\gamma_2\star\gamma_1)(G)/\gamma_2(G)=\gamma_1(G/\gamma_2(G))$, we see by $(5)$ of Remark \ref{rem21}    that $$\gamma_1((N\gamma_2(G)\cap (\gamma_2\star\gamma_1)(G))/\gamma_2(G))=(N\gamma_2(G)\cap (\gamma_2\star\gamma_1)(G))/\gamma_2(G).$$
Note that
\begin{multline*}
  (N\gamma_2(G)\cap (\gamma_2\star\gamma_1)(G))/\gamma_2(G)=\\
 (N\cap (\gamma_2\star\gamma_1)(G))\gamma_2(G)/\gamma_2(G)\simeq
 (N\cap (\gamma_2\star\gamma_1)(G))/(N\cap\gamma_2(G))\\
=(N\cap (\gamma_2\star\gamma_1)(G))/\gamma_2(N)\trianglelefteq N/\gamma_2(N).
\end{multline*}
It means that $(N\cap (\gamma_2\star\gamma_1)(G))/\gamma_2(N)\subseteq \gamma_1(N/\gamma_2(N))$. Thus  $N\cap (\gamma_2\star\gamma_1)(G)\subseteq  (\gamma_2\star\gamma_1)(N)$, i.e $\gamma_2\star\gamma_1$ satisfies $(F3)$.\end{proof}

Here we introduce the height $h_\gamma(G)$ of a group $G$ which corresponds to a given functorial $\gamma$.

\begin{defn}
  Let $\gamma$ be a functorial. Then the $\gamma$-series of $G$ is defined  starting from $\gamma_{(0)}(G)=1$, and then by induction $\gamma_{(i+1)}(G)=(\gamma_{(i)}\star\gamma)(G)$ is the inverse image of $\gamma(G/\gamma_{(i)}(G))$. The least number $h$ such that $\gamma_{(h)}(G)=G$ is defined to be $\gamma$-height $h_\gamma(G)$ of $G$. If there is no such number, then $h_\gamma(G)=\infty$.
\end{defn}

The following Lemma directly follows from Proposition \ref{length0}.

\begin{lem}\label{n}
 Let $\gamma$   be a functorial. If $\gamma$  satisfies $(F1)$ and $(F2)$, then $\gamma_{(n)}$ satisfies $(F1)$ and $(F2)$ for all natural $n$. Moreover if $\gamma$ satisfies $(F3)$, then  $\gamma_{(n)}$ satisfies $(F3)$ for all natural $n$.
\end{lem}

\begin{lem}\label{ineq}
Let $\gamma$   be a functorial. If $\gamma$  satisfies $(F1)$ and $(F2)$, then
 $h_\gamma(G/N)\leq h_\gamma(G)\leq h_\gamma(N)+h_\gamma(G/N)$ for every $N\trianglelefteq G$. Moreover, if  $\gamma$ also satisfies $(F3)$, then  $h_\gamma(N)\leq h_\gamma(G)$.
\end{lem}

\begin{proof}
Note that $\gamma_{(n)}$ satisfies $(F1)$ and $(F2)$ for every $n$ by Lemma \ref{n}.

 Since $\gamma_{(n)}$ satisfies   $(F1)$, $G/N=\gamma_{h_\gamma(G)}(G)/N\leq \gamma_{(h_\gamma(G))}(G/N)\leq G/N$. So $\gamma_{(h_\gamma(G))}(G/N)=G/N$. Thus $h_\gamma(G/N)\leq h_\gamma(G)$.

 Since $\gamma_{(n)}$ satisfies   $(F2)$, we see that $N=\gamma_{(h_\gamma(N))}(N)\subseteq \gamma_{(h_\gamma(N))}(G)$. Note that $h_\gamma(G/\gamma_{(h_\gamma(N))}(G))\leq h_\gamma(G/N)$. Thus  $h_\gamma(G)\leq h_\gamma(N)+h_\gamma(G/N)$.

 Assume that $\gamma$ also satisfies $(F3)$. Then  $\gamma_{(n)}$ satisfies $(F3)$ by Lemma \ref{n}. Now $N=G\cap N=\gamma_{(h_\gamma(G))}(G)\cap N\subseteq\gamma_{(h_\gamma(G))}(N)\leq N$. So $\gamma_{(h_\gamma(G))}(N)=N$. Thus $h_\gamma(N)\leq h_\gamma(G)$.
\end{proof}

If $\gamma=\mathrm{F}^*$, then $h_\gamma(G)=h^*(G)$ for every group $G$.  The non-$p$-soluble length can also be defined with the help of functorials. Here by $R_p(G)$ we denote the $p$-soluble radical of a group $G$.

\begin{lem}\label{lambdap}
Let $\mathrm{\overline{F}_p}=R_p\star \mathrm{F}^*\star R_p$ and $G$ be a non-$p$-soluble group. Then $\lambda_p(G)$ is the smallest natural $i$ with  $\mathrm{\overline{F}_p}_{(i)}(G)=G$.
\end{lem}

\begin{proof}
Let $1=G_0\leq G_1\leq\dots\leq G_{2h+1}=G$ be the shortest normal series in which for $i$ odd the factor $G_{i+1}/G_i$ is $p$-soluble (possibly trivial),
and for $i$ even the factor $G_{i+1}/G_i$ is a (non-empty) direct product of nonabelian simple
groups.

Note that $G_1\leq R_p(G)$ and $G_2/G_1$ is quasinilpotent. Hence $G_2R_p(G)/R_p(G)$ is quasinilpotent. It means that $G_2R_p(G)/R_p(G)\leq \mathrm{F}^*(G/R_p(G))$.    Hence $G_2\leq (R_p\star \mathrm{F}^*)(G)$. Since $G_3/G_2$ is $p$-soluble, we see that $G_3(R_p\star \mathrm{F}^*)(G)/(R_p\star \mathrm{F}^*)(G)$ is $p$-soluble. Hence  $G_3(R_p\star \mathrm{F}^*)(G)/(R_p\star \mathrm{F}^*)(G)\leq R_p(G/(R_p\star \mathrm{F}^*)(G))$. It means that $G_3\leq \mathrm{\overline{F}_p}(G)=\mathrm{\overline{F}_p}_{(1)}(G)$.

Assume that we proved $G_{2i+1}\leq \mathrm{\overline{F}_p}_{(i)}(G)$. Let prove that  $G_{2(i+1)+1}\leq \mathrm{\overline{F}_p}_{(i+1)}(G)$.

From $G_{2i+1}\leq \mathrm{\overline{F}_p}_{(i)}(G)$ it follows that $G_{2i+1}\leq (\mathrm{\overline{F}_p}_{(i)}\star R_p)(G)$.
Note that $G_{2i+2}/G_{2i+1}$ is quasinilpotent. It means that $G_{2i+2}(\mathrm{\overline{F}_p}_{(i)}\star R_p)(G)/(\mathrm{\overline{F}_p}_{(i)}\star R_p)(G)$ is quasinilpotent. Hence $G_{2i+2}\leq ((\mathrm{\overline{F}_p}_{(i)}\star R_p)\star\mathrm{F}^*)(G)$. Since
$G_{2(i+1)+1}/G_{2i+2}$ is $p$-soluble, we see that $G_{2(i+1)+1}(\mathrm{\overline{F}_p}_{(i)}\star R_p\star\mathrm{F}^*)(G)/(\mathrm{\overline{F}_p}_{(i)}\star R_p\star \mathrm{F}^*)(G)$ is $p$-soluble.
Hence  $G_{2(i+1)+1}(\mathrm{\overline{F}_p}_{(i)}\star R_p\star \mathrm{F}^*)(G)/((\mathrm{\overline{F}_p}_{(i)}\star R_p\star \mathrm{F}^*)(G)\leq R_p(G/(\mathrm{\overline{F}_p}_{(i)}\star R_p\star \mathrm{F}^*)(G))$. It means that $G_{2(i+1)+1}\leq (\mathrm{\overline{F}_p}_{(i)}\star R_p\star \mathrm{F}^*\star R_p)(G)=\mathrm{\overline{F}_p}_{(i+1)}(G)$.

Therefore $\lambda_p(G)\geq n$ where $n$ is the smallest integer with   $\mathrm{\overline{F}_p}_{(n)}(G)=n$. Since $R_p\star R_p=R_p$, we see that  $\mathrm{\overline{F}_p}_{(n)}(G)$ presents a normal series $1\leq F_1\leq F_2\leq\dots\leq F_{2n+1}$ in which for $i$ odd the factor $F_{i+1}/F_i$ is $p$-soluble (possibly trivial),
and for $i$ even the factor $F_{i+1}/F_i$ is a (non-empty) direct product of nonabelian simple
groups. So $\lambda_p(G)\leq n$. Thus $\lambda_p(G)= n$.\end{proof}

Now we are able to estimate the $\gamma$-height of the direct product subgroups and of the join of subnormal subgroups:

\begin{thm}\label{len3}
Let $\gamma$ be  a functorial with $\gamma(H)>1$ for every group $H$ that satisfies $(F1)$ and~$(F2)$.

$(1)$  If $G=\times_{i=1}^n  A_i$ is the direct product of its normal subgroups $A_i$, then $h_\gamma(G)= \max\{h_\gamma(A_i)\mid 1\leq i\leq n\}$.

$(2)$ Let $G=\langle A_i\mid 1\leq i\leq n\rangle$ be the join of its subnormal subgroups $A_i$. Then $h_\gamma(G)\leq \max\{h_\gamma(A_i)\mid 1\leq i\leq n\}$. If $ \gamma$ satisfies $(F3)$, then $h_\gamma(G)= \max\{h_\gamma(A_i)\mid 1\leq i\leq n\}$.
\end{thm}

\begin{proof}
Note that $\gamma_{(n)}$ satisfies $(F1)$ and $(F2)$ for every $n$ by Proposition \ref{length0}.

$(1)$ From Lemma \ref{FF} it follows that if   $G=\times_{i=1}^n  A_i$, then $\gamma_{(n)}(G)=\times_{i=1}^n  \gamma_{(n)}(A_i)$. It means that $h_\gamma(G)= \max\{h_\gamma(A_i)\mid 1\leq i\leq n\}$.

$(2)$ Assume  that $G=\langle A_i\mid 1\leq i\leq n\rangle$ is the join of its subnormal subgroups $A_i$,  $h_1=\max\{h_\gamma(A_i)\mid 1\leq i\leq n\}$ and $h_2=h_\gamma(G)$. Since $\gamma_{(n)}$ satisfies $(F2)$, we see that  $\gamma_{(n)}(N)\subseteq \gamma_{(n)}(G)$ for every subnormal subgroup $N$ of $G$ and every $n$. Now $$G=\langle A_i\mid 1\leq i\leq n\rangle=\langle \gamma_{(h_1)}(A_i)\mid 1\leq i\leq n\rangle\subseteq \gamma_{(h_1)}(G)\subseteq G.$$ Hence $\gamma_{(h_1)}(G)=G$. It means that $h_2\leq h_1$.

  Suppose that $\gamma$ satisfies $(F3)$. Now  $\gamma_{(n)}$ satisfies $(F3)$
  for every $n$ by Proposition \ref{length0}. From $(5)$ of Remark \ref{rem21} it follows that $\gamma_{(n)}(G)\cap N=\gamma_{(n)}(N)$   for every subnormal subgroup $N$ of $G$. Now   $A_i=A_i\cap G=A_i\cap \gamma_{(h_2)}(G)=\gamma_{(h_2)}(A_i)$. It means that $h_\gamma(A_i)\leq h_2$ for every $i$. Hence $h_1\leq h_2$. Thus $h_1=h_2$.
\end{proof}

\begin{cor}
  Let a group $G=\langle A_i\mid 1\leq i\leq n\rangle$ be the join of its subnormal subgroups $A_i$. Then $h^*(G)= \max\{h^*(A_i)\mid 1\leq i\leq n\}$ and $\lambda_p(G)= \max\{\lambda_p(A_i)\mid 1\leq i\leq n\}$.
\end{cor}

\section{The Classes of Groups Method}

 Recall that a \emph{formation} is a class $\mathfrak{F}$ of groups with the following properties:
$(a)$ every homomorphic image of an $\mathfrak{F}$-group is an $\mathfrak{F}$-group, and
$(b)$ if $G / M$ and $G / N$ are $\mathfrak{F}$-groups, then also $G/(M\cap N)\in \mathfrak{F}$.
Recall that the $ \mathfrak{F}$-\emph{residual} of a group $ G$ is the smallest normal subgroup $G^\mathfrak{F}$ of $ G$ with $ G/G^\mathfrak{F}\in\mathfrak{F}$.

A formation is called Fitting if $(a)$ from $N\trianglelefteq G\in\mathfrak{F}$ it follows that $N\in\mathfrak{F}$ and $(b)$  a group $G\in\mathfrak{F}$ whenever it is a product of normal $\mathfrak{F}$-subgroups. Recall that the $ \mathfrak{F}$-\emph{radical} $G_\mathfrak{F}$ of a group $ G$ is the greatest normal $\mathfrak{F}$-subgroup.

The classes $\mathfrak{N}^*$ of all quasinilpotent groups and $\mathfrak{S}^p$ of all $p$-soluble groups are Fitting formations.

From \cite[IX, Remarks 1.11 and Theorem 1.12]{s8} and \cite[IV,  Theorem 1.8]{s8} follows

\begin{lem}
Let    $\mathfrak{F}$ and $\mathfrak{H}$ be  non-empty Fitting formations. Then
$$\mathfrak{FH}=(G\mid G^\mathfrak{F}\in\mathfrak{H})=(G\mid G/G_\mathfrak{H}\in\mathfrak{F})$$
is a Fitting formation.
\end{lem}

\begin{cor}
  The class   $ \mathfrak{H}_p=(G\mid \overline{\mathrm{F}}_p(G)=G)$  is a Fitting formation.
\end{cor}

It is straightforward to check that for a Fitting formation $\mathfrak{F}$, the $\mathfrak{F}$-radical can be considered as a functorial $\gamma$ which satisfies $(F1)$, $(F2)$ and $(F3)$. For convenience in this case denote $ h_\gamma$ by $h_\mathfrak{F}$.
Now $h^*(G)=h_{\mathrm{F}^*}(G)=h_{\mathfrak{N}^*}(G)$ and for a non-$p$-soluble group $\lambda_p(G)=h_{\overline{\mathrm{F}}_p}(G)=h_{\mathfrak{H}_p}(G)$.

\begin{lem}\label{minusone}
Let $\mathfrak{F}$ be a Fitting formation.
  If $H\neq 1$ and $h_\mathfrak{F}(H)<\infty$, then $h_\mathfrak{F}(H^{\mathfrak{F}})=h_\mathfrak{F}(H)-1$.
\end{lem}

\begin{proof}
  Let prove that if $H\neq 1$ and $h_\mathfrak{F}(H)<\infty$, then $h_\mathfrak{F}(H^{\mathfrak{F}})=h_\mathfrak{F}(H)-1$. Let $h_\mathfrak{F}(H)=n$ and $h_\mathfrak{F}(H^{\mathfrak{F}})=k$. Then ${H_\mathfrak{F}}_{(n-1)}(H)<H$ and $H/{H_\mathfrak{F}}_{(n-1)}\in\mathfrak{F}$. It means that $H^{\mathfrak{F}}\leq {H_\mathfrak{F}}_{(n-1)}$. Since ${H_\mathfrak{F}}_{(n-1)}$ satisfies $(F3)$, we see that ${(H^{\mathfrak{F}})_\mathfrak{F}}_{(n-1)}=H^{\mathfrak{F}}$. Hence  $k\leq n-1$.

 Note that $H^{\mathfrak{F}}={(H^{\mathfrak{F}})_\mathfrak{F}}_{(k)}\leq {H_\mathfrak{F}}_{(k)}$. It means that $H/{H_\mathfrak{F}}_{(k)}\in\mathfrak{F}$. Hence $k\geq n-1$. Thus $k=n-1$.
\end{proof}

If $\mathfrak{F},\mathfrak{H},\mathfrak{K}\neq\emptyset$ are formations, then  $(\mathfrak{FH})\mathfrak{K}=\mathfrak{F}(\mathfrak{HK})$  by \cite[IV, Theorem 1.8]{s8}. That is why the class $\mathfrak{F}^n=\underbrace{\mathfrak{F}\dots\mathfrak{F}}_n$ is a well defined formation.

\begin{lem}\label{form0}
  For a natural number $n$    and a Fitting formation $\mathfrak{F}$ holds   $\mathfrak{F}^n=(G\mid h_\mathfrak{F}(G)\leq n)$.
\end{lem}

\begin{proof}
From Lemma \ref{minusone} it follows that if $G\in(G\mid h_\mathfrak{F}(G)\leq n)$, then $G^{\mathfrak{F}^n}=1$. It means that $(G\mid h_\mathfrak{F}(G)\leq n)\subseteq \mathfrak{F}^n$. Assume that there is a group $G\in\mathfrak{F}^n$ with $h_\mathfrak{F}(G)>n$. Note that $\overline{G}^\mathfrak{F}\neq \overline{G}$ for every quotient group $\overline{G}\not\simeq 1$ of $G$.   Then $h_\mathfrak{F}(G^{\mathfrak{F}^n})>0$ by Lemma \ref{minusone}. It means that $G^{\mathfrak{F}^n}\neq 1$, a contradiction. Therefore $\mathfrak{F}^n\subseteq (G\mid h_\mathfrak{F}(G)\leq n)$. Thus $\mathfrak{F}^n=(G\mid h_\mathfrak{А}(G)\leq n)$.
\end{proof}

In the next lemma we recall the key properties of mutually permutable products

\begin{lem}\label{MP}
  Let a group $G=AB$ be a mutually permutable product of subgroups $A$ and $B$. Then

  $(1)$   \cite[Lemma 4.1.10]{PFG}  $G/N=(AN/N)(BN/N)$ is a mutually permutable product of subgroups $AN/N$ and $BN/N$ for every normal subgroup $N$ of $G$.

  $(2)$ \cite[Lemma 4.3.3(4)]{PFG} If $N$ is a minimal normal subgroup of a group $G$, then $\{N\cap A, N\cap B\}\subseteq \{1, N\}$.

  $(3)$ \cite[Lemma 4.3.3(5)]{PFG} If $N$ is a minimal normal subgroup of $G$ contained in $A$ and $B \cap N = 1$, then
$N\leq C_G(A)$   or $N\leq C_G(B)$. If furthermore $N$ is not cyclic, then $N\leq C_G(B)$.

$(4)$ \cite[Theorem 4.3.11]{PFG} $A_GB_G\neq 1$.

$(5)$  \cite[Corollary 4.1.26]{PFG} $A'$ and $B'$ are subnormal in $G$.
\end{lem}

Recall that $\pi(G)$ is the set of all prime divisors of  $|G|$,  $\pi(\mathfrak{F})=\underset{G\in\mathfrak{F}}\cup\pi(G)$ and $\mathfrak{N}_\pi$ denote the class of all nilpotent $\pi$-groups.

\begin{lem}\label{bound}
Let $\mathfrak{F}$ be a Fitting formation. Assume that
$h_\mathfrak{F}(G)\leq h+1$
 for every mutually permutable product $G$ of two $\mathfrak{F}$-subgroups. Then
$$\max\{h_\mathfrak{F}(A), h_\mathfrak{F}(B)\}-1\leq h_\mathfrak{F}(G)\leq\max\{h_\mathfrak{F}(A), h_\mathfrak{F}(B)\}+ h$$ for every mutually permutable product $G$ of two subgroups $A$ and $B$ with $h_\mathfrak{F}(A), h_\mathfrak{F}(B)<\infty$.
\end{lem}

\begin{proof}
  If $A=1$ or $B=1$, then there is nothing to prove. Assume that $A, B\neq 1$.
  Let a group $G=AB$ be the product of mutually permutable subgroups $A$ and $B$. From $h_\mathfrak{F}(A), h_\mathfrak{F}(B)<\infty$ it follows that $\pi(G)\subseteq\pi(\mathfrak{F})$. According to \cite[IX, Lemma 1.8]{s8}  $\mathfrak{N}_{\pi(\mathfrak{F})}\subseteq\mathfrak{F}$. Note that $A'$ and $B'$ are subnormal in $G$ by $(5)$ of Lemma \ref{MP}.  Since $H^{\mathfrak{F}}\trianglelefteq H^{\mathfrak{N}_{\pi(\mathfrak{F})}}\trianglelefteq H'$ holds for every $\pi(\mathfrak{F})$-group $H$, subgroups $A^{\mathfrak{F}}$ and $B^{\mathfrak{F}}$ are subnormal in $G$. Let $C=\langle A^{\mathfrak{F}}, B^{\mathfrak{F}}\rangle^G=\langle\{(A^{\mathfrak{F}})^x\mid x\in G\}\cup\{(B^{\mathfrak{F}})^x\mid x\in G\}\rangle$. Then by $(2)$ of Theorem \ref{len3} and by Lemma \ref{minusone}
\begin{multline*}
  h_\mathfrak{F}(C)=\max\left\{\{(h_\mathfrak{F}(A^{\mathfrak{F}})^x)\mid x\in G\}\cup\{(h_\mathfrak{F}(B^{\mathfrak{F}})^x)\mid x\in G\}\right\}\\=\max\{h_\mathfrak{F}(A^{\mathfrak{F}}), h_\mathfrak{F}(B^{\mathfrak{F}})\}=\max\{h_\mathfrak{F}(A), h_\mathfrak{F}(B)\}-1.
\end{multline*}
Now $G/C=(AC/C)(BC/C)$ is a mutually permutable product of $\mathfrak{F}$-subgroups $AC/C$ and $BC/C$ by
$(1)$ of Lemma \ref{MP}. It means that $h_\mathfrak{F}(G/C)\leq h+1$ by our assumption. With the help of Lemma \ref{ineq} we see that $$h_\mathfrak{F}(G)\leq h_\mathfrak{F}(C)+h_\mathfrak{F}(G/C)\leq  \max\{h_\mathfrak{F}(A), h_\mathfrak{F}(B)\}-1+1+h=\max\{h_\mathfrak{F}(A), h_\mathfrak{F}(B)\}+h.$$
From the other hand, $h_\mathfrak{F}(G)\geq h_\mathfrak{F}(C)=\max\{h_\mathfrak{F}(A), h_\mathfrak{F}(B)\}-1$  by $(2)$ of Theorem \ref{len3}.
\end{proof}

\begin{lem}\label{bound1}
Let $\mathfrak{F}$ be a Fitting formation. Assume that a group $G$ is the least order group with

$(1)$  $G$ is a mutually permutable product of two subgroups $A$ and $B$ with $h_\mathfrak{F}(A)\geq h_\mathfrak{F}(B)$;

$(2)$ $h_\mathfrak{F}(G)=h_\mathfrak{F}(A)-1$.

Then $G$ has the unique minimal normal subgroup $N$,    $N\leq A$ and  $h_\mathfrak{F}(A/N)=h_\mathfrak{F}(A)-1$.
 \end{lem}

\begin{proof}
  Let $N$ be a minimal normal subgroup of $G$. Then $ N\cap A\in\{N, 1\}$ by $(2)$ of Lemma \ref{MP}.

Assume that $N\cap A=1$. Now $ G/N=(AN/N)(BN/N)$ is a mutually permutable product of groups  $AN/N$ and $BN/N$ by $(1)$ of Lemma \ref{MP}. By our assumption  and $h_\mathfrak{F}(G)\geq h_\mathfrak{F}(G/N)\geq h_\mathfrak{F}(AN/N)=h_\mathfrak{F}(A)$, a contradiction. Hence $N\cap A=N$ for every minimal normal subgroup $N$ of $G$.

 Now  $h_\mathfrak{F}(G)+1=h_\mathfrak{F}(A)>h_\mathfrak{F}(G)\geq h_\mathfrak{F}(G/N)\geq h_\mathfrak{F}(A/N)\geq h_\mathfrak{F}(A)-1$. It means that $h_\mathfrak{F}(G)=h_\mathfrak{F}(A/N)= h_\mathfrak{F}(A)-1$.

 If $G$ has two  minimal normal subgroups $N_1$ and $N_2$, then $h_\mathfrak{F}(A/N_1)=h_\mathfrak{F}(A/N_2)= h_\mathfrak{F}(A)-1$. It means $h_\mathfrak{F}(A)<h_\mathfrak{F}(A)-1$ by Lemma \ref{form0}, a contradiction. Hence $G$ has a unique minimal normal subgroup $N$.
\end{proof}

\section{Proof of Theorem \ref{Main1}(1)}

Our proof relies on the notion of the $\mathfrak{X}$-hypercenter.
A chief factor $H/K$ of  $G$ is called   $\mathfrak{X}$-\emph{central} in $G$ provided    $$(H/K)\rtimes (G/C_G(H/K))\in\mathfrak{X}$$ (see \cite[p. 127--128]{s6} or \cite[1, Definition 2.2]{Guo2015}). A normal subgroup $N$ of $G$ is said to be $\mathfrak{X}$-\emph{hypercentral} in $G$ if $N=1$ or $N\neq 1$ and every chief factor of $G$ below $N$ is $\mathfrak{X}$-central. The symbol $\mathrm{Z}_\mathfrak{X}(G)$ denotes the $\mathfrak{X}$-\emph{hypercenter} of $G$, that is, the product of all normal $\mathfrak{X}$-\emph{hypercentral} in $G$ subgroups. According to \cite[Lemma 14.1]{s6} or \cite[1, Theorem 2.6]{Guo2015} $\mathrm{Z}_\mathfrak{X}(G)$ is the largest normal $\mathfrak{X}$-hypercentral subgroup of $G$. If $\mathfrak{X}=\mathfrak{N}$ is the class of all nilpotent groups, then $\mathrm{Z}_\mathfrak{N}(G)=\mathrm{Z}_\infty(G)$ is the hypercenter of $G$.

\begin{lem}\label{form}
  Let $n$ be a natural number.  Then   $(\mathfrak{N}^*)^n=(G\mid h^*(G)\leq n)=(G\mid G=\mathrm{Z}_{(\mathfrak{N}^*)^n}(G))$.
\end{lem}

\begin{proof}
 First part follows from Lemma \ref{form0}. It is well known that the class of all quasinilpotent groups is a composition (or Baer-local, or solubly saturated) formation (see \cite[Example 2.2.17]{s9}).  According to  \cite[Theorem 7.9]{s6} $(\mathfrak{N}^*)^n $ is a composition  formation. Now $(\mathfrak{N}^*)^n=(G\mid G=\mathrm{Z}_{(\mathfrak{N}^*)^n}(G))$ by \cite[1, Theorem 2.6]{Guo2015}.
\end{proof}

For a normal section $H/K$ of  $G$ the subgroup $C_G^*(H/K)=HC_G(H/K)$ is called an \emph{inneriser} (see \cite[Definition 1.2.2]{s9}). It is the set of all elements of $G$ that induce inner automorphisms on $H/K$. From the definition of the generalized Fitting subgroup it follows that it is the intersection of  innerisers of all chief factors.

\begin{lem}\label{lem4}
  Let $N$ be a normal subgroup of a group $G$. If $N$ is a direct product of isomorphic simple groups and $h^*(G/C_G^*(N))\leq k-1$, then $\mathrm{F}^*_{(k)}(G/N)=\mathrm{F}^*_{(k)}(G)/N$.
\end{lem}

\begin{proof}
Assume that $h^*(G/C_G^*(N))\leq k-1$.
Let $F/N=\mathrm{F}^*_{(k)}(G/N)$. Then $\mathrm{F}^*_{(k)}(G)\subseteq F$. Now $F/C_F^*(N)\simeq FC_G^*(N)/C_G^*(N)\trianglelefteq G/C_G^*(N)$.  Therefore $h^*(F/C_F^*(N))\leq k-1$. It means that $h^*(F/C_F^*(H/K))\leq k-1$ for every chief factor $H/K$ of $F$ below $N$. Hence $(H/K)\rtimes (F/C_F(H/K))\in (\mathfrak{N}^*)^k$  for every chief factor $H/K$ of $F$ below $N$. It means that $N\leq \mathrm{Z}_{(\mathfrak{N}^*)^k}(F)$. Thus $F\in (\mathfrak{N}^*)^k$ by Lemma \ref{form}. So $F\subseteq \mathrm{F}^*_{(k)}(G)$. Thus $\mathrm{F}^*_{(k)}(G)= F$.
\end{proof}

\begin{lem}\label{quasi2}
  If a group $G=AB$ is a product of mutually permutable quasinilpotent subgroups $A$ and $B$, then $h^*(G)\leq 2$.
\end{lem}

\begin{proof}
To prove this lemma we need only to prove that if a group $G=AB$ is a product of mutually permutable quasinilpotent subgroups $A$ and $B$, then $G\in(\mathfrak{N}^*)^2$ by Lemma \ref{form}.
 Assume the contrary. Let $G$ be a minimal order counterexample.

$(1)$ \emph{$G$ has a unique minimal normal subgroup $N$ and    $G/N\in (\mathfrak{N}^*)^2$.}

Note that $G/N$ is a mutually permutable product of quasinilpotent subgroups $(AN/N)$ and $(BN/N)$ by $(1)$ of Lemma \ref{MP}. Hence $G/N\in (\mathfrak{N}^*)^2$ by our assumption. Since $(\mathfrak{N}^*)^2$ is a formation, we see that $G$ has a unique minimal normal subgroup.
According to $(4)$ of Lemma \ref{MP}  $A_GB_G\neq 1$.  WLOG we may assume that $G$ has a minimal normal subgroup  $N\leq A$.

$(2)$ $N\leq A\cap B$.

Suppose that $N\cap B=1$. Then $A\leq C_G(N)$ or $B\leq C_G(N)$ by $(3)$ of Lemma \ref{MP}.
If $A\leq C_G(N)$, then $N\rtimes G/C_G(N)\simeq N\rtimes B/C_B(N)\in(\mathfrak{N}^*)^2$.
If $B\leq C_G(N)$, then $N\rtimes G/C_G(N)\simeq N\rtimes A/C_A(N)\in(\mathfrak{N}^*)\subseteq(\mathfrak{N}^*)^2$ by \cite[Corollary 2.2.5]{s9}. In both cases $N\leq \mathrm{Z}_{(\mathfrak{N}^*)^2}(G)$. It means that $G\in(\mathfrak{N}^*)^2$, a contradiction.
Now $N\cap B\neq 1$. Hence $N\leq A\cap B$ by $(2)$ of Lemma \ref{MP}.

$(3)$ \emph{$N$ is non-abelian.}

Assume that $N$ is abelian. Since $A$ is quasinilpotent, we see that   $A/C_A(N)$ is a $p$-group. By analogy $B/C_B(N)$ is a $p$-group. Note that $A/C_A(N)\simeq AC_G(N)/C_G(N)$ and $B/C_B(N)\simeq BC_G(N)/C_G(N)$. From $G=AB$ it follows that $G/C_G(N)$ is a $p$-group. Since $N$ is a chief factor of $G$, we see that       $G/C_G(N)\simeq 1$. So $N\leq\mathrm{Z}_\infty(G)\leq \mathrm{Z}_{(\mathfrak{N}^*)^2}(G)$. Thus $G\in(\mathfrak{N}^*)^2$, a contradiction.
It means that $N$ is non-abelian.

$(4)$ \emph{The final contradiction.}

Now $N$ is a direct product of minimal normal subgroups of $A$. Since $A$ is quasinilpotent, we see that every element of $A$ induces an inner automorphism on every  minimal normal subgroup of $A$. Hence     every element of $A$ induces an inner automorphism on $N$. By analogy  every element of $B$ induces an inner automorphism on $N$. From $G=AB$ it follows that  every element of $G$ induces an inner automorphism on $N$. So $NC_G(N)=G$ or $G/C_G(N)\simeq N$. Now $N\rtimes(G/C_G(N))\in(\mathfrak{N}^*)^2$. It means that $N\leq \mathrm{Z}_{(\mathfrak{N}^*)^2}(G)$. Thus $G\in(\mathfrak{N}^*)^2$ and $h^*(G)\leq 2$, the final contradiction.
\end{proof}

\begin{proof}[Proof of Theorem \ref{Main1}(1)]
  Let a group $G$ be a mutually permutable product of subgroups $A$ and $B$. From Theorem \ref{len3} and Lemma  \ref{quasi2} it follows that
$$\max\{h^*(A), h^*(B)\}-1\leq h^*(G)\leq \max\{h^*(A), h^*(B)\}+1.$$

Assume that $\max\{h^*(A), h^*(B)\}-1= h^*(G)$. WLOG let $h^*(A)=h^*(G)-1$. We may assume that a group $G$ is the least order group with such properties. Then  $G$ has the unique minimal normal subgroup $N$,    $N\leq A$ and  $h^*(A/N)=h^*(A)-1$ by Lemma \ref{bound1}.

 Assume that $h^*(A/C_A^*(N))<h^*(A)-1$. Then $$\mathrm{F}_{(h^*(A)-1)}^*(A/N)=\mathrm{F}_{(h^*(A)-1)}^*(A)/N<A/N$$ by Lemma \ref{lem4}. It means that $h^*(A)=h^*(A/N)$, a contradiction. Hence  $h^*(A/C_A^*(N))=h^*(A)-1$.

 Since $G/C_G^*(N)=(AC_G^*(N)/C_G^*(N))(BC_G^*(N)/C_G^*(N))$ is a mutually permutable products of subgroups $AC_G^*(N)/C_G^*(N)$ and $BC_G^*(N)/C_G^*(N)$ by $(1)$ of Lemma \ref{MP} and $A/C_A^*(N)\simeq AC_G^*(N)/C_A^*(N)$, we see that $h^*(G/C_G^*(N))\geq h^*(A/C_A^*(N))=h^*(A)-1$ by our assumptions. Note that $\mathrm{F}^*(G)\leq C_G^*(N)$. Now $h^*(G)-1=h^*(G/\mathrm{F}^*(G))\geq h^*(G/C_G^*(N))\geq h^*(A/C_A^*(N))=h^*(A)-1$. It means that $h^*(G)\geq h^*(A)$, the final contradiction.
      \end{proof}

\section{Proof of Theorem \ref{Main1}(2)}

\begin{lem}\label{Hp}
Let $p$ be a prime and $\mathfrak{H}=\mathfrak{H}_p$.     If a group $G=AB$ is a product of mutually permutable $\mathfrak{H}$-subgroups $A$ and $B$, then $G\in\mathfrak{H}$.\end{lem}

\begin{proof}
 Assume the contrary. Let $G$ be a minimal order counterexample.

$(1)$ \emph{$G$ has a unique minimal normal subgroup $N$,   $G/N\in \mathfrak{H}$ and $N$ is not $p$-soluble.}

Note that $G/N$ is a mutually permutable product of $\mathfrak{H}$-subgroups $(AN/N)$ and $(BN/N)$ by $(1)$ of Lemma \ref{MP}. Hence $G/N\in \mathfrak{H}$ by our assumption. Since $\mathfrak{H}$ is a formation, we see that $G$ has a unique minimal normal subgroup.
According to $(4)$ of Lemma \ref{MP}  $A_GB_G\neq 1$.  WLOG we may assume that $G$ has a minimal normal subgroup  $N\leq A$.

If $N$ is $p$-soluble, then $\mathrm{\overline{F}}_p(G)/N=\mathrm{\overline{F}}_p(G/N)=G$, i.e. So $\mathrm{\overline{F}}_p(G)=G$. Thus $G\in\mathfrak{H}$, a contradiction.

$(2)$ $N\leq A\cap B$.

Suppose that $N\cap B=1$. Note that $N$ is not cyclic by $(1)$.  Then   $B\leq C_G(N)$ by $(3)$ of Lemma \ref{MP}. Hence $N\rtimes G/C_G(N)\simeq N\rtimes A/C_A(N)\in\mathfrak{H}$ by \cite[Corollary 2.2.5]{s9}. It means that $N\leq \mathrm{Z}_{\mathfrak{H}}(G)$. Therefore $G\in \mathfrak{H}$, a contradiction.
Now $N\cap B\neq 1$. Hence $N\leq A\cap B$ by $(2)$ of Lemma \ref{MP}.

$(4)$ \emph{The final contradiction.}

Since $N$ is the unique minimal normal subgroup of $G$ and non-abelian, we see that $C_G(N)=1$. So $C_A(N)=C_B(N)=1$. Hence $R_p(A)=R_p(B)=1$. In particular $\mathrm{F}(A)=\mathrm{F}(B)=1$. Note that all minimal normal subgroups of $A$ are in $N$. For $B$ is  the same situation. Thus $N=\mathrm{F}^*(A)=\mathrm{F}^*(B)$. So $G/N$ is a mutually permutable product of $p$-soluble groups. Since the class of all $p$-soluble groups is closed by extensions by $p$-soluble groups, $G/N$ is $p$-soluble by  $(1)$ and  $(4)$ of Lemma \ref{MP}.
 From $N\leq \mathrm{F}^*(G)$ it follows that $G\in\mathfrak{H}$, the contradiction.
\end{proof}

\begin{proof}[Proof of Theorem \ref{Main1}(2)]
Let $\mathfrak{H}=\mathfrak{H}_p$ and a group $G$ be a mutually permutable product of subgroups $A$ and $B$. First we a going to prove that $\max\{h_\mathfrak{H}(A), h_\mathfrak{H}(B)\}=h_\mathfrak{H}(G)$.

By Lemmas \ref{bound} and \ref{quasi2} we have
$$\max\{h_\mathfrak{H}(A), h_\mathfrak{H}(B)\}-1\leq h_\mathfrak{H}(G)\leq \max\{h_\mathfrak{H}(A), h_\mathfrak{H}(B)\}.$$

Assume that $\max\{h_\mathfrak{H}(A), h_\mathfrak{H}(B)\}-1= h_\mathfrak{H}(G)$ for some mutually permutable product $G$ of $A$ and $B$. Assume that $G$ is a minimal order group with this property.  WLOG let $h_\mathfrak{H}(A)=h_\mathfrak{H}(G)-1$.  Then  $G$ has the unique minimal normal subgroup $N$,    $N\leq A$ and  $h_\mathfrak{H}(A/N)=h_\mathfrak{H}(A)-1$ by Lemma \ref{bound1}.

If $N$ is  $p$-soluble, then $R_p(A/N)=R_p(A)/N$. It means that $\overline{\mathrm{F}}_p(A/N)=\overline{\mathrm{F}}_p(A)/N$. Thus $h_\mathfrak{H}(A/N)=h_\mathfrak{H}(A)$, a contradiction.

It means that $R_p(G)=1$. Note that now $N$ is a simple non-abelian group. Since $N$ is a unique minimal normal subgroup of $G$, we see that  $N=\mathrm{F}^*(G)$. Now $h_\mathfrak{H}(G/N)=h_\mathfrak{H}(G)-1$. Therefore
$$ h_\mathfrak{H}(G)-1=h_\mathfrak{H}(G/N)\geq h_\mathfrak{H}(A/N)=h_\mathfrak{H}(A)-1.$$
Thus $h_\mathfrak{H}(G)\geq h_\mathfrak{H}(A)$, the contradiction.

 We proved that $\max\{h_\mathfrak{H}(A), h_\mathfrak{H}(B)\}= h_\mathfrak{H}(G)$.

 Let $G$ be a mutually permutable product of groups $A$ and $B$. If $A, B$ are $p$-soluble, then $G$ is $p$-soluble by $(1)$ and $(4)$ of Lemma \ref{MP}. Hence $\lambda_p(G)=\lambda_p(A)=\lambda_p(B)=0$. Now assume that at least one of subgroups $A, B$ is not $p$-soluble. Then  $G$ is not $p$-soluble by $(1)$ and $(4)$ of Lemma \ref{MP}. WLOG let $h_\mathfrak{H}(A)\geq h_\mathfrak{H}(B)$. Hence $A$ is not $p$-soluble. We proved that $h_\mathfrak{H}(A)=h_\mathfrak{H}(G)$. Note that $h_\mathfrak{H}(G)=\lambda_p(G)$, $h_\mathfrak{H}(A)=\lambda_p(A)$,  $h_\mathfrak{H}(B)= \lambda_p(B)$ if $B$ is not $p$-soluble   by Lemma  \ref{lambdap} and $0=\lambda_p(B)<1=h_\mathfrak{H}(B)\leq h_\mathfrak{H}(A)=\lambda_p(A)$ otherwise. Thus
$\max\{\lambda_p(A), \lambda_p(B)\}=\lambda_p(G)$.      \end{proof}

\section{Non-Frattini length}

The Frattini subgroup $\Phi(G)$ play an important role in the theory of classes of groups.  One of the useful properties of the Fitting subgroup of a soluble group is that it is strictly greater than the Frattini subgroup of the same group. Note that the generalized Fitting subgroup is non-trivial in every group but there are groups in which it coincides with the Frattini subgroup.
That is why the following length seems interesting.

\begin{defn}
Let  $1=G_0\leq G_1\leq\dots\leq G_{2h}=G$ be a shortest normal series in which for $i$ even $G_{i+1}/G_i\leq \Phi(G/G_i)$,
and for $i$ odd the factor $G_{i+1}/G_i$ is a (non-empty) direct product of simple
groups. Then $h=\tilde{h}(G)$ will be called the non-Frattini length of a group $G$.
\end{defn}

Note that if $G$ is a soluble group, then $\tilde{h}(G)=h(G)$.  Another reason that leads us to this length   is the generalization of the Fitting subgroup   $\tilde{\mathrm{F}}(G)$ introduced by P.~Schmid \cite{Schmid1972} and L.A. Shemetkov
 \cite[Definition~7.5]{f4} and defined by
 $$ \Phi(G)\subseteq \tilde{\mathrm{F}}(G) \textrm{ and
  } \tilde{\mathrm{F}}(G)/\Phi(G)=\mathrm{Soc}(G/\Phi(G)).$$
P. F\"orster \cite{Foerster1985} showed that $\tilde{\mathrm{F}}(G)$ can be defined by
 $$ \Phi(G)\subseteq \tilde{\mathrm{F}}(G) \textrm{ and
  } \tilde{\mathrm{F}}(G)/\Phi(G)=\mathrm{F}^*(G/\Phi(G)).$$
Let $\Phi$ and  $\tilde{\mathrm{F}}$ be functorials that assign $\Phi(G)$ and $\tilde{\mathrm{F}}(G)$  to every group $G$. Then $\tilde{\mathrm{F}}=\Phi\star\mathrm{F}^*$. It is well known that $\Phi$ satisfies $(F1)$ and $(F2)$.    Hence $\tilde{\mathrm{F}}$  satisfies $(F1)$ and $(F2)$ by Proposition \ref{length0}.

Note that $\Phi(G/\Phi(G))\simeq 1$. By analogy with the proof of Lemma \ref{lambdap} one can show that the non-Frattini length $\tilde{h}(G)$ of a group $G$ and  $h_{\mathrm{\tilde{F}}}(G)$ coincide for every group $G$. The following theorem shows connections between  the non-Frattini length and the generalized Fitting height.

\begin{thm}\label{len1}
 For any group $G$ holds $\tilde{h}(G)\leq h^*(G)\leq 2\tilde{h}(G)$.   There exists a group $H$ with $\tilde{h}(H)=n$ and $h^*(H)=2n$ for any natural $n$.
\end{thm}

\begin{proof}
   Since $\Phi(G)$ and $\mathrm{Soc}(G/\Phi(G))$ are quasinilpotent, we see that $\mathrm{F}^*(G)\leq\mathrm{\tilde F}(G)\leq \mathrm{F}_{(2)}^*(G)$.  Now $\mathrm{F}_{(n)}^*(G)\leq\mathrm{\tilde F}_{(n)}(G)\leq \mathrm{F}_{(2n)}^*(G)$. Hence if $\mathrm{\tilde F}_{(n)}(G)=G$, then $\mathrm{F}_{(n)}^*(G)\leq G$ and $\mathrm{F}_{(2n)}^*(G)=G$. It means $\tilde{h}(G)\leq h^*(G)\leq 2\tilde{h}(G)$.

Let $K$ be a group, $K_1$ be isomorphic to the regular wreath product of  $\mathbb{A}_5$ and $K$. Note that the base $B$  of it is the unique minimal normal subgroup of $K_1$ and non-abelian. According to \cite{Griess1978}, there is a
 Frattini $\mathbb{F}_3K_1$-module $A$ which is faithful for $K_1$ and a Frattini extension  $A\rightarrowtail K_2\twoheadrightarrow K_1$
such that $A\stackrel {K_1}{\simeq} \Phi(K_2)$ and $K_2/\Phi(K_2)\simeq K_1$.

Let denote $K_2$ by $\mathbf{f}(K)$.
Now $\mathbf{f}(K)/\mathrm{\tilde F}(\mathbf{f}(K))\simeq K$. From the definition of $h_{\mathrm{\tilde F}}=\tilde{h}$ it follows that $\tilde{h}(\mathbf{f}(K))=\tilde{h}(K)+1$.

Note that $\Phi(\mathbf{f}(K))\subseteq \mathrm{F}^*(\mathbf{f}(K))$. Assume that $\Phi(\mathbf{f}(K))\neq \mathrm{F}^*(\mathbf{f}(K))$. It means that $\mathrm{F}^*(\mathbf{f}(K))=\mathrm{\tilde F}(\mathbf{f}(K))$ is quasinilpotent. By \cite[X, Theorem 13.8]{19} it follows that $\Phi(\mathbf{f}(K))\subseteq \mathrm{Z}(\mathrm{F}^*(\mathbf{f}(K)))$. It means that $1<B\leq C_{K_1}(A)$. Thus $A$ is not faithful, a contradiction.

Thus  $\Phi(\mathbf{f}(K))= \mathrm{F}^*(\mathbf{f}(K))$ and $\mathbf{f}(K)/\mathrm{F}^*(\mathbf{f}(K))\simeq K_1$. Since $K_1$ has a unique minimal normal subgroup $B$ and it is non-abelian, we see that $\mathrm{F}^*(K_1)=B$. It means that $\mathbf{f}(K)/\mathrm{F}^*_{(2)}(\mathbf{f}(K))\simeq K$. From the definition of $h^*$ it follows that $h^*(\mathbf{f}(K))=h^*(K)+2$.

As usual, let $\mathbf{f}^{(1)}(K)=\mathbf{f}(K)$ and $\mathbf{f}^{(i+1)}(K)=\mathbf{f}(\mathbf{f}^{(i)}(K))$. Then   $\tilde{h}(\mathbf{f}^{(n)}(1))=n$ and $h^*(\mathbf{f}^{(n)}(1))=2n$ for any natural $n$.
\end{proof}

The following proposition directly follows from Theorem \ref{len3}.

\begin{prop}
 Let a group $G=\langle A_i\mid 1\leq i\leq n\rangle$ be the join of its subnormal subgroups $A_i$. Then $\tilde{h}(G)\leq \max\{\tilde{h}(A_i)\mid 1\leq i\leq n\}$.
 \end{prop}

One of the main differences between the non-Frattini length and the generalized Fitting height is that the non-Frattini length of a normal subgroup can be greater than the non-Frattini length of a group.

\begin{example}
Let $E\simeq \mathbb{A}_5$. There is an $\mathbb{F}_5E$-module $V$ such
 that $R=Rad(V)$ is a faithful irreducible $\mathbb{F}_5E$-module and $V/R$ is
 an irreducible trivial $\mathbb{F}_5E$-module (how to construct such module, for example, see \cite{Murashka2022}).  Let $G=V\leftthreetimes E$. Now $\Phi(G)=R$ by \cite[B, Lemma 3.14]{s8}. Note that $G/\Phi(G)=G/R\simeq Z_5\times E$. So $\mathrm{\tilde{F}}(G)=G$ and  $\tilde{h}(G)=1$. Note that $G=V(RE)$ where $V$ and $RE$ are normal subgroups of $G$. Since $ V$ is abelian, we see that $\tilde{h}(V)=1$. Note that $R$ is a unique minimal normal subgroup of $RE$ and $\Phi(RE)=1$. It means that  $\mathrm{\tilde{F}}(RE)=R$ and $\tilde{h}(RE)=2$. Thus $\tilde{h}(G)<\max\{\tilde{h}(V), \tilde{h}(RE)\}$ and $\tilde{\mathrm{F}}$ does not satisfy $(F3)$.
\end{example}

 Recall \cite[Definition 4.1.1]{PFG} that a group $G$ is called a totally permutable product of its subgroups $A$ and $B$ if $G=AB$ and every subgroup of $A$ permutes with every subgroup of $B$.

\begin{thm}\label{nonF}
  Let a group $G=AB$ be a totally permutable product of subgroups $A$ and $B$. Then  $$\max\{\tilde{h}(A), \tilde{h}(B)\}-1\leq \tilde{h}(G)\leq \max\{\tilde{h}(A), \tilde{h}(B)\}+1. $$
\end{thm}

\begin{proof}
If $A=1$ or $B=1$, then  $\max\{\tilde{h}(A), \tilde{h}(B)\}=\tilde{h}(G)$. Assume that $A, B\neq 1$.

According to \cite[Proposition 4.1.16]{PFG} $A\cap B\leq \mathrm{F}(G)$. Hence $A\cap B\leq \mathrm{F}^*(G)$. Now $\overline{G}=G/\mathrm{F}^*(G)$ is a totally permutable product of $\overline{A}=A\mathrm{F}^*(G)/\mathrm{F}^*(G)$ and $\overline{B}=B\mathrm{F}^*(G)/\mathrm{F}^*(G)$ by \cite[Corollary 4.1.11]{PFG}. Note that  $\overline{A}\cap \overline{B}\simeq 1$. According to \cite[Lemma 4.2.2]{PFG} $[\overline{A}, \overline{B}]\leq \mathrm{F}(\overline{G})$. So $[\overline{A}, \overline{B}]\leq \mathrm{F}^*(\overline{G})$. It means that
   $$ \overline{G}/\mathrm{F}^*(\overline{G})=(\overline{A}\mathrm{F}^*(\overline{G})/\mathrm{F}^*(\overline{G}))\times (\overline{B}\mathrm{F}^*(\overline{G})/\mathrm{F}^*(\overline{G})). $$

Note that for the formation $\mathfrak{U}$ of all supersoluble groups we have $\mathfrak{U}\subset \mathfrak{N}^2\subset(\mathfrak{N}^*)^2$. Hence if $H=H_1H_2$ is a product of totally permutable $\mathfrak{(\mathfrak{N}^*)}^2$-subgroups $H_1$ and $H_2$, then $H\in\mathfrak{(\mathfrak{N}^*)}^2$ by \cite[Theorem 5.2.1]{PFG}. Analyzing the proof of \cite[Theorem 5.2.2]{PFG} we see that this theorem is true not only for saturated formation, but for formations $\mathfrak{F}=(G\mid G=\mathrm{Z}_\mathfrak{F}(G))$. In particular, it is true for  $(\mathfrak{N}^*)^2$. Thus  if $H=H_1H_2\in\mathfrak{(\mathfrak{N}^*)}^2$ is a product of totally permutable subgroups $H_1$ and $H_2$, then $H_1, H_2\in\mathfrak{(\mathfrak{N}^*)}^2$. Now $(\mathfrak{N}^*)^2$ satisfies conditions of \cite[Proposition 5.3.9]{PFG}.

Therefore $A\cap \mathrm{F}^*_{(2)}(G)=\mathrm{F}^*_{(2)}(A)$ and $B\cap \mathrm{F}^*_{(2)}(G)=\mathrm{F}^*_{(2)}(B)$. Note that

$$\overline{A}\mathrm{F}^*(\overline{G})/\mathrm{F}^*(\overline{G})\simeq A\mathrm{F}^*_{(2)}(G)/\mathrm{F}^*_{(2)}(G)\simeq A/\mathrm{F}^*_{(2)}(A). $$

By analogy $ \overline{B}\mathrm{F}^*(\overline{G})/\mathrm{F}^*(\overline{G})\simeq B/\mathrm{F}^*_{(2)}(B)$. Hence

$$G/\mathrm{F}^*_{(2)}(G)\simeq (A/\mathrm{F}^*_{(2)}(A))\times (B/\mathrm{F}^*_{(2)}(B)).$$

  By Theorem \ref{len3} and $\tilde{h}=h_{\tilde{\mathrm{F}}}$  we have $\tilde{h}(G/\mathrm{F}^*_{(2)}(G))=
  \max\{\tilde{h}(A/\mathrm{F}^*_{(2)}(A)),\tilde{h}(B/\mathrm{F}^*_{(2)}(B))\}.$

From $\tilde{\mathrm{F}}(H)\leq \mathrm{F}^*_{(2)}(H)\leq\tilde{\mathrm{F}}_{(2)}(H)$ and Lemma \ref{ineq} it follows that for any group $H\neq 1$ holds
$$\tilde{h}(H)-1=\tilde{h}(H/\tilde{\mathrm{F}}(H))\geq \tilde{h}(H/\mathrm{F}^*_{(2)}(H)) \geq \tilde{h}(H/\tilde{\mathrm{F}}_{(2)}(H))\geq\tilde{h}(H)-2.$$
Therefore
  $$\{\tilde{h}(G)-\tilde{h}(G/\mathrm{F}^*_{(2)}(G)),
 \tilde{h}(A)-\tilde{h}(A/\mathrm{F}^*_{(2)}(A)),
 \tilde{h}(B)-\tilde{h}(B/\mathrm{F}^*_{(2)}(B))\}\subseteq \{1, 2\}.$$
Thus $\max\{\tilde{h}(A), \tilde{h}(B)\}-1\leq \tilde{h}(G)\leq \max\{\tilde{h}(A), \tilde{h}(B)\}+1$.
\end{proof}

While proving Theorem \ref{nonF} we were not able to answer  the following question:

\begin{quest}
   Let a group $G=AB$ be a totally permutable product of subgroups $A$ and $B$. Is $\max\{\tilde{h}(A), \tilde{h}(B)\}\leq \tilde{h}(G)$?
\end{quest}

The following question seems interesting

\begin{quest}
Do there exists a constant $h$ with $|\max\{\tilde{h}(A), \tilde{h}(B)\}- \tilde{h}(G)|\leq h$
  for any  mutually permutable product $G=AB$   of subgroups $A$ and $B$?
\end{quest}

D.A. Towers \cite{TOWERS2017} defined and studied analogues of $\mathrm{F}^*(G)$ and $\mathrm{\tilde F}(G)$ for  Lie algebras. Using these subgroups and the radical (of a Lie algebra) one can introduce the generalized Fitting height, the non-soluble length and the non-Frattini length of a (finite dimension) Lie algebra.

\begin{quest}
Estimate the generalized Fitting height, the non-soluble length and the non-Frattini length of a $($finite dimension$)$ Lie algebra that is the sum of its two subalgebras $($ideals, subideals, mutually or totally permutable subalgebras$)$.
\end{quest}

\bibliographystyle{plain}
\bibliography{length}

\end{document}